\documentclass[11pt,a4paper]{article}
\usepackage{}
\usepackage{amsfonts}
\usepackage{multirow}
 \usepackage{epsfig}
 \usepackage{epstopdf}
\usepackage[T1]{fontenc}
\usepackage{geometry}
\usepackage{amsbsy,amsmath,latexsym,amsfonts, epsfig, color, authblk, amssymb, graphics, bm}
\usepackage{epsf,slidesec,epic,eepic}
\usepackage{fancybox}
\usepackage{fancyhdr}
\usepackage{setspace}
\usepackage{nccmath}
\usepackage[colorlinks, citecolor=blue]{hyperref}

\newtheorem{theorem}{Theorem}[section]
\newtheorem{lemma}{Lemma}[section]

\newtheorem{proposition}{Proposition}[section]

\newtheorem{remark}{Remark}[section]

\newenvironment{proof}{{\noindent \bf Proof:}}{\hfill$\Box$\medskip}

\definecolor{lred}{rgb}{1,0.8,0.8}
\definecolor{lblue}{rgb}{0.8,0.8,1}
\definecolor{dred}{rgb}{0.6,0,0}
\definecolor{dblue}{rgb}{0,0,0.5}
\definecolor{dgreen}{rgb}{0,0.5,0.5}

 \title{Error bounds for rank constrained optimization problems and applications\footnote{This work is supported by the National Natural Science Foundation of China under project No.11571120, the Natural Science Foundation of Guangdong Province under project No.2015A030313214
 and No.2015A030310298, and the Fundamental Research Funds for the Central Universities(SCUT).}}

\author{Shujun Bi\footnote{State Key Laboratory of Scientific and Engineering Computing, Institute of Computational Mathematics and
 Scientific/Engineering Computing, Academy of Mathematics and Systems Science, Chinese Academy of Science, Beijing, P. R. China (bishujun@lsec.cc.ac.cn).}
 \ \ {\rm and}\ \ Shaohua Pan\footnote{Corresponding author. Department of Mathematics, South China University of Technology, Tianhe District, Guangzhou 510641, P. R. China (shhpan@scut.edu.cn).}}

\begin{document}

  \maketitle

 \begin{abstract}
  This paper is concerned with the rank constrained optimization problem whose feasible set
  is the intersection of the rank constraint set $\mathcal{R}=\!\big\{X\in\mathbb{X}\ |\ {\rm rank}(X)\le \kappa\big\}$
  and a closed convex set $\Omega$. We establish the local (global) Lipschitzian type error bounds
  for estimating the distance from any $X\in \Omega$ ($X\in\mathbb{X}$) to the feasible set
  and the solution set, respectively, under the calmness of a multifunction associated to
  the feasible set at the origin, which is specially satisfied by three classes of
  common rank constrained optimization problems. As an application of the local Lipschitzian
  type error bounds, we show that the penalty problem yielded by moving the rank constraint
  into the objective is exact in the sense that its global optimal solution set coincides with
  that of the original problem when the penalty parameter is over a certain threshold.
  This particularly offers an affirmative answer to the open question whether the penalty problem
  (32) in \cite{GS-Major} is exact or not. As another application,
  we derive the error bounds of the iterates generated by a multi-stage convex relaxation approach
  to those three classes of rank constrained problems
  and show that the bounds are nonincreasing as the number of stages increases.


 \medskip
 \noindent
 {\bf AMS Subject Classifictions (2010): 90C26, 90C22}
 \end{abstract}

 \medskip

  \section{Introduction}\label{sec1}

  Let $\mathbb{X}$ denote the vector space $\mathbb{R}^{n_1\times n_2}$ of all
  $n_1\times n_2$ real matrices or the vector space $\mathbb{H}^n$ of all
  $n\times n$ Hermitian matrices, both endowed with the trace inner product
  $\langle \cdot,\cdot\rangle$ and its induced norm $\|\cdot\|_F$.
  Given a positive integer $\kappa$ and a suitable continuous loss function
  $f\!:\mathbb{X}\to\mathbb{R}$, we are concerned with the following rank constrained
  optimization problem
  \begin{align}\label{prob}
    \min_{X\in\mathbb{X}}\Big\{f(X)\ | \ {\rm rank}(X)\leq \kappa,\ X\in \Omega\Big\},
  \end{align}
  where $\Omega$ is a convex compact subset of $\mathbb{X}$. Such a problem has many
  applications in a host of fields including statistics, signal and image processing,
  system identification and control, collaborative filtering, quantum tomography, finance, and so on
  (see, e.g., \cite{Tony-ROP,Fazel02,Fazel03,GS-Major,Mesbahi,Recht10,Salakhutdinov10, Zhang-LAA}).
  In the sequel, we denote by $\mathcal{F}$ the feasible set of \eqref{prob} and assume that
  $\mathcal{F}\ne \emptyset$, which implies that the solution set of \eqref{prob},
  denoted by $\mathcal{F}^*$, is nonempty.

  \medskip

  Due to the combinatorial nature of rank function, problem \eqref{prob} is generally NP-hard,
  and it is almost impossible to seek a global optimal solution. A common way to resolve
  this class of problems is to adopt convex relaxation technique, which yields a desirable
  local optimal even feasible solution by solving a single or a sequence of tractable convex
  optimization problems. The popular nuclear norm convex relaxation method proposed in \cite{Fazel02}
  belongs to the single-stage convex relaxation class, which received active research
  in the past several years from many fields such as optimization, statistics, information, computer science,
  and so on (see, e.g., \cite{Tony-ROP,Candes09,Gross11,Liu12,Recht10, Recht-NSP,Toh10}).
  Although the nuclear norm promotes low rank solutions, it has a big difference from the rank function
  in a general setting since the former is convex whereas the latter is nonconvex even concave
  (in the positive semidefinite cone). Hence, when the set $\Omega$ is characterizing
  some structure conflicted with low rank, such as the correlation or density matrix structure,
  the nuclear norm relaxation method will fail in yielding a low rank solution.
  In view of this, many researchers recently develop effective solution methods based on
  the sequential convex relaxation models arising from the penalty problems \cite{GS-Major,LuZh-PDrank},
  the nonconvex surrogate problems \cite{Fazel03,LXW13,Miao13,Mohan12}, and the rank constrained
  optimization problem itself \cite{Meka-SVP,SGS-ICML,Wen-SOR}. We notice that to measure the distance
  from any given point to the feasible set or the solution set plays a key role in the analysis of these methods.
  Motivated by this, we in this work take the first step towards the study on Lipschitzian type error bounds
  for \eqref{prob}.

  \medskip

  In this paper, we show that the calmness of a multifunction associated to the feasible set
  $\mathcal{F}$ at the origin is a sufficient and necessary condition for the local Lipschitzian
  error bounds to estimate the distance from any $X\in\Omega$ to $\mathcal{F}$, which is specially
  satisfied by three classes of common rank constrained optimization problems \eqref{prob} where
  $\Omega$ is a ball set, a density matrix set or a correlation matrix set, and under this condition
  derive the global error bound for estimating the distance from any $X\in\mathbb{X}$ to $\mathcal{F}$.
  In addition, under an additional mild assumption for the objective function $f$,
  we also establish the local (global) Lipschitzian error bounds for estimating the distance
  from any $X\in\Omega$ ($X\in\mathbb{X}$) to the solution set $\mathcal{F}^*$.
  To the best of our knowledge, this paper is the first one to study
  the Lipschitzian type error bounds for low-rank optimization problems,
  though there are many works on error bounds for the system of linear inequalities
  and (nondifferentiable) convex inequalities (see, e.g., \cite{Luo96,Man-94,Man-MP98,Pang97-Exp,Zhang}
  and references therein). To illustrate the potential applications of the derived error bounds,
  we show that the penalty problem
  \begin{align}\label{penalty-rc}
    \min_{X\in \Omega}\,\Big\{f(X)+\rho\!\sum_{i=\kappa+1}^n\!\sigma_i(X)\Big\}
  \end{align}
  is exact in the sense that its global optimal solution set coincides with that of \eqref{prob}
  when the penalty parameter $\rho$ is over a certain threshold. This does not only affirmatively
  answer the open question proposed in \cite{GS-Major} about whether the penalty problem (32)
  there is exact or not for the rank constrained correlation matrix problem, but also provides
  a platform for designing convex relaxation algorithms for \eqref{prob}. In addition, we also establish
  the error bounds of the iterates generated by a multi-stage convex relaxation approach to
  the rank constrained optimization problems with $\Omega$ being a ball set, a density matrix set
  or a correlation matrix set, and show that the error bound sequence is nonincreasing.

  \medskip

  To close this section, we provide a brief summary of notations used in this paper:
  \begin{itemize}
   \item  $\mathbb{H}^{n}_+$ denotes the cone of Hermitian positive semidefinite matrices.
          For $X\in\mathbb{H}^{n}$, we assume that its eigenvalue decomposition takes the form of
          $X=\sum_{i=1}^n \lambda_i(X)u_iu_i^{\mathbb{T}}$ where $\lambda_1(X)\ge\cdots\ge\lambda_n(X)$
          and all $u_i$ are complex orthonormal column vectors.

   \item  For $X\in\mathbb{R}^{n_1\times n_2}$, we denote by $\|X\|_*$ the nuclear norm of $X$, and assume that
          its singular value decomposition (SVD) takes the form of $X=\sum_{i=1}^n \sigma_i(X)u_iv_i^{\mathbb{T}}$
          where $\sigma_1(X)\ge\cdots\ge\sigma_n(X)$ with $n=\min(n_1,n_2)$, and all $u_i\in\mathbb{R}^{n_1}$
          and $v_i\in\mathbb{R}^{n_2}$ are orthonormal column vectors.

   \item  For a given vector $x\in\mathbb{R}^n$, ${\rm Diag}(x)$ denotes the diagonal matrix with $i$th diagonal entry being $x_i$.
          For a vector $u\in\mathbb{R}_{++}^n$, $\frac{1}{\sqrt{u}}$ means the vector
          $(\frac{1}{\sqrt{u_1}},\ldots,\frac{1}{\sqrt{u_n}})^{\mathbb{T}}\in\mathbb{R}^n$.

   \item The notation $\mathbb{B}_{\mathbb{X}}$ denotes a closed unit ball of the space $\mathbb{X}$ centered at the origin.
         For a closed subset $\mathcal {S}\subseteq\mathbb{X}$, $\Pi_\mathcal {S}(X)$ means a projection point of $X$
         onto the set $\mathcal {S}$.
  \end{itemize}

  \section{Lipschitzian type error bounds}\label{sec2}

  Let $\mathbb{Y}$ and $\mathbb{Z}$ be the finite dimensional vector spaces
  which are both endowed with the norm $\|\cdot\|$. Recall that a multifunction
  $\Upsilon\!:\mathbb{Y}\rightrightarrows\mathbb{Z}$ is calm at $\overline{y}$
  for $\overline{z}\in \Upsilon(\overline{y})$ if there exist a constant
  $\alpha\ge 0$ along with neighborhoods $\mathcal{U}$ of $\overline{y}$ and
  $\mathcal{V}$ of $\overline{z}$ such that
  \[
    \Upsilon(y)\cap \mathcal{V}\subseteq \Upsilon(\overline{y})+\alpha\|y-\overline{y}\|\mathbb{B}_{\mathbb{Z}}
    \quad {\rm for \ all}\ y\in \mathcal{U}.
  \]
  By \cite[Exercise 3H.4]{DR09}, we know that there is no need at all to mention
  a neighborhood $\mathcal{U}$ of $\overline{y}$ in the description of calmness,
  i.e., the following equivalent description holds.
  \begin{lemma}\label{Lemma2.1}
   For a multifunction $\Upsilon\!:\mathbb{Y}\rightrightarrows\mathbb{Z}$,
   the calmness of $\Upsilon$ at $\overline{y}$ for $\overline{z}\in\Upsilon(\overline{y})$
   is equivalent to the existence of a constant $\alpha\ge 0$ and a neighborhood $\mathcal{V}$ of $\overline{z}$ such that
  \begin{equation*}
    \Upsilon(y)\cap \mathcal{V}\subseteq \Upsilon(\overline{y})+\alpha\|y-\overline{y}\|\mathbb{B}_{\mathbb{Z}}
    \quad {\rm \ for \ all}\ y\in \mathbb{Y},
  \end{equation*}
  or the existence of a constant $\alpha\ge 0$ and a neighborhood $\mathcal{V}$ of $\overline{z}$ such that
  \begin{equation*}
   {\rm dist}(z, \Upsilon(\overline{y}))\le \alpha\,{\rm dist}(\overline{y},\Upsilon^{-1}(z))
    \quad {\rm \ for \ all}\ z\in \mathcal{V}.
  \end{equation*}
  \end{lemma}
  Define a multifunction $\Gamma\!: \mathbb{R}\rightrightarrows\mathbb{X}$ associated to
  the feasible set of problem \eqref{prob} as follows
  \begin{equation}\label{Gamma}
    \Gamma(\omega):=\Big\{X\in \Omega\ |\ \sum_{i=\kappa+1}^n\sigma_i(X)=\omega\Big\}\quad{\rm for}\ \omega\in\mathbb{R}.
  \end{equation}
  In this section, we show that the calmness of the multifunction $\Gamma$ at $0$ for each $X\in\Gamma(0)$
  is a sufficient and necessary condition for the local Lipschitzian error bounds to estimate the distance
  from any $X\in\Omega$ to the feasible set $\mathcal{F}$, which is specially satisfied for three classes
  of popular rank constrained optimization problems, and then establish the global error bounds for estimating
  the distance from any $X\in \mathbb{X}$ to $\mathcal{F}$ under the calmness of $\Gamma$, and derive
  the local and global Lipschitzian error bounds for estimating the distance to $\mathcal{F}^*$
  under an additional restricted strong convexity for the objective function $f$.

  \subsection{Local error bounds}\label{subsec2.1}

  The following theorem states that the distance from any $Z\in \Omega$ to $\Gamma(0)=\Omega\cap\mathcal{R}$
  can be bounded above by $\sum_{i=\kappa+1}^n\sigma_i(Z)$ iff the multifunction $\Gamma$ is calm at $0$
  for each $X\in\Gamma(0)$.
  \begin{theorem}\label{theorem2.1}
   Let $\Gamma$ be the multifunction defined by \eqref{Gamma}. The multifunction $\Gamma$
   is calm at $0$ for each $X\in\Gamma(0)$ if and only if there exists a constant $c>0$ such that
   \begin{equation}\label{ineq-lemma}
     {\rm dist}(Z,\Gamma(0))\le c\,{\rm dist}(0,\Gamma^{-1}(Z))=c\, \sum_{i=\kappa+1}^n\sigma_i(Z)
     \quad{\rm for\ all}\ Z\in\Omega.
   \end{equation}
  \end{theorem}
  \begin{proof}
   ``$\Longrightarrow$''. By the calmness of $\Gamma$ at $0$ for each $X\in\Gamma(0)=\Omega\cap\mathcal{R}$
   and Lemma \ref{Lemma2.1}, it follows that for each $X\in \Omega\cap\mathcal{R}$,
   there exist constants $\alpha(X)\ge0$ and $\epsilon(X)>0$ such that
   \begin{equation}\label{subregular}
     {\rm dist}(Y,\Gamma(0))\le \alpha(X)\,{\rm dist}(0,\Gamma^{-1}(Y))\quad\ \forall\,Y\in \mathbb{B}(X,\epsilon(X)),
   \end{equation}
   where $\mathbb{B}(X,\epsilon(X))$ is a closed ball of radius $\epsilon(X)$ centered at $X$.
   Notice that the compact set $\Omega\cap\mathcal{R}$ is covered by the set $\bigcup_{X\in\Omega\cap\mathcal{R}}\big(X+\frac{\epsilon(X)}{2}\mathbb{B}_{\mathbb{X}}^{\circ}\big)$,
   where $\mathbb{B}_{\mathbb{X}}^{\circ}$ denotes the open unit ball around the origin in $\mathbb{X}$.
   By the Heine-Borel theorem, there exist a finite number of points
   $X^1,X^2,\ldots,X^m\in \Omega\cap\mathcal{R}$ such that
   \(
     \Omega\cap\mathcal{R}\subseteq\bigcup_{i=1}^m\big(X^i+\frac{\epsilon(X^i)}{2}\mathbb{B}_{\mathbb{X}}^{\circ}\big).
   \)
   Write
   \[
     \overline{\epsilon}:=\min\{\epsilon(X^1),\ldots,\epsilon(X^m)\}\ \ {\rm and}\ \
    \overline{\alpha}:=\max\{\alpha(X^1),\ldots,\alpha(X^m)\}.
   \]
   Let $Z$ be an arbitrary point from $\Omega$. We proceed the arguments by two cases as below.

   \medskip
   \noindent
   {\bf Case 1:} ${\rm dist}(Z,\Omega\cap\mathcal{R})\le \overline{\epsilon}/2$. Since the set $\Omega\cap\mathcal{R}$ is closed,
   there must exist $\overline{Z}\in \Omega\cap\mathcal{R}$ such that $\|Z-\overline{Z}\|_F\le \overline{\epsilon}/2$.
   Since $\overline{Z}\in \Omega\cap\mathcal{R}$, there exists a $k\in\{1,2,\ldots,m\}$ such that
   $\|\overline{Z}\!-\!X^k\|_F\!<\!\epsilon(X^k)/2$. Consequently,
   \(
     \|Z\!-\!X^k\|_F\!\le\! \|Z\!-\!\overline{Z}\|_F\!+\!\|\overline{Z}\!-\!X^k\|_F\!\le\!\epsilon(X^k).
   \)
   Together with \eqref{subregular},
   \(
     {\rm dist}(Z,\Gamma(0))\le \overline{\alpha}\,{\rm dist}(0,\Gamma^{-1}(Z)).
   \)
   This shows that \eqref{ineq-lemma} holds with $c=\overline{\alpha}$.

   \medskip
   \noindent
   {\bf Case 2:} ${\rm dist}(Z,\Omega\cap\mathcal{R})>\overline{\epsilon}/2$. Now there must exist
   an $\eta>0$ such that $\sum_{i=\kappa+1}^n\sigma_i(Y)\ge\eta$ for all $Y\in \Omega$ with
   ${\rm dist}(Y,\Omega\cap\mathcal{R})>\overline{\epsilon}/2$. If not, one may select a sequence
   $\{Z^k\}\subseteq \Omega$ with ${\rm dist}(Z^k,\Omega\cap\mathcal{R})>\overline{\epsilon}/2$
   such that $\sum_{i=\kappa+1}^n\sigma_i(Z^k)\le\eta^k$ for all $k$, where $\{\eta^k\}$ is a sequence of
   positive numbers with $\lim_{k\to +\infty}\eta^k=0$.
   Since $\Omega$ is compact, we without loss of generality assume that $\{Z^k\}$ converges to $Z^*\in \Omega$.
   Then, from the locally Lipschitz continuity of $\sigma_i(\cdot)$, it follows that
   $\sum_{i=\kappa+1}^n\sigma_i(Z^*)\le 0$, and then $Z^*\in \Omega\cap\mathcal{R}$. On the other hand,
   from ${\rm dist}(Z^k,\Omega\cap\mathcal{R})>\overline{\epsilon}/2$ for all $k$, we have
   \(
      {\rm dist}(Z^*,\Omega\cap\mathcal{R})>\overline{\epsilon}/2.
   \)
  Thus, we obtain a contradiction, and the above statement holds.
  Since $\Omega$ is bounded, it follows that ${\rm dist}(\cdot,\Omega\cap\mathcal{R})$ is bounded above on $\Omega$, say,
  by some $M>0$. Thus, for all $Z\in \Omega$ with ${\rm dist}(Z,\Omega\cap\mathcal{R})>\overline{\epsilon}/2$, one has that
  \(
     {\rm dist}(Z,\Omega\cap\mathcal{R})\le M\le ({M}/{\eta}){\textstyle \sum_{i=\kappa+1}^n}\sigma_i(Z).
  \)
  By taking $c=M/\eta$, the desired inequality \eqref{ineq-lemma} then follows.

  \medskip
  \noindent
  ``$\Longleftarrow$. Let $X$ be an arbitrary point from $\Gamma(0)$ and
  $\epsilon\in(0,1)$ be an arbitrary constant. By Lemma \ref{Lemma2.1}, we only need to
  argue that there must exist a constant $c'>0$ such that
  \begin{equation}\label{temp-equa21}
    {\rm dist}(Z,\Gamma(0))\le c'\,{\rm dist}(0,\Gamma^{-1}(Z))\quad\ \forall Z\in\mathbb{B}(X,\epsilon).
  \end{equation}
  Indeed, for any $Z\in\mathbb{B}(X,\epsilon)$, if $Z\notin \Omega$,
  then $\Gamma^{-1}(Z)=\emptyset$ by noting that ${\rm dom}\,\Gamma^{-1}\subseteq\Omega$,
  and inequality \eqref{temp-equa21} holds for any $c'>0$; if $Z\in\Omega$, then
  by taking $c'=c$, inequality \eqref{temp-equa21} follows directly from \eqref{ineq-lemma}.
  Until now, the proof is completed.
 \end{proof}

 \medskip

  Next, we show that the sufficient and necessary condition in Theorem \ref{theorem2.1}
  is especially satisfied by three classes of common rank constrained optimization problems
  in which $\Omega$ is a ball, a density matrix set or a correlation matrix set,
  and provide the Lipschitzian type error bounds to estimate the distance
  from any $X\in\Omega$ to $\Omega\cap\mathcal{R}$ for them.

  \medskip
  
  {\bf (1) Rank constrained optimization problems over a ball.}  The feasible set of this class of
  rank constrained optimization problems takes the following form
  \begin{equation}\label{F-ball}
    \mathcal{F}:=\Big\{X\in\mathbb{R}^{n_1\times n_2}\ |\ {\rm rank}(X)\le\kappa,\ |\!\|X\|\!|\le \gamma\Big\},
  \end{equation}
  where $|\!\|\cdot\|\!|$ is a matrix norm and $\gamma>0$ is a given constant.
  The following proposition provides a Lipschitzian type bound for the error
  ${\rm dist}(X,\mathcal {F})$ with $|\!\|X\|\!|\le \gamma$.
 \begin{proposition}\label{prop-ball}
  Let $\Omega=\big\{X\in \mathbb{R}^{n_1\times n_2}\ |\ |\!\|X\|\!|\le \gamma\big\}$.
  Then, for any $X\in\Omega$, by assuming that $X$ has the SVD of the form
  $\sum_{i=1}^n\sigma_i(X)u_iv_i^\mathbb{T}$, it holds that
  \begin{equation}\label{ineq-ball}
   {\rm dist}(X,\mathcal {F})\le \sqrt{1+c^2_u/c_l^2}\,\big\|X-\Pi_{\mathcal{R}}(X)\big\|_F
   \ \ {\rm with}\ \ \Pi_\mathcal {R}(X):=\sum_{i=1}^\kappa\sigma_i(X)u_iv_i^\mathbb{T},
  \end{equation}
  where $c_l$ and $c_u$ are positive constants such that $c_l\|\cdot\|_F\leq|\!\|\cdot\|\!|\leq c_u\|\cdot\|_F$.
  In particular,
  \begin{equation}\label{inclusion-ball}
   \Gamma(t)\subseteq \Gamma(0)+\sqrt{1+c^2_u/c_l^2}\,|t|\,\mathbb{B}_{\mathbb{R}^{n_1\times n_2}}\quad\ \forall t\in\mathbb{R}.
  \end{equation}
  When $|\!\|\cdot\|\!|$ is unitarily invariant, the above constant $\sqrt{1+c^2_u/c_l^2}$
  can be replaced by $1$.
 \end{proposition}
 \begin{proof}
 Let $X$ be an arbitrary point in $\Omega$ with the SVD given by $\sum_{i=1}^n\sigma_i(X)u_iv_i^\mathbb{T}$. Define
  \begin{equation}\label{XF}
    \widehat{X}_{\mathcal{F}}
    := \frac{\gamma}{\max(\gamma,|\!\|\Pi_\mathcal {R}(X)\|\!|)}\Pi_\mathcal {R}(X).
  \end{equation}
  It is easy to verify that $\widehat{X}_{\mathcal{F}}\in \mathcal{F}=\Gamma(0)$.
  Thus, to establish \eqref{ineq-ball}, it suffices to argue that
  \begin{equation}\label{aim-ineq1}
    \|X-\widehat{X}_{\mathcal{F}}\|_F\le \sqrt{1+c^2_u/c_l^2}\,\big\|X-\Pi_{\mathcal{R}}(X)\big\|_F.
  \end{equation}
  When $|\!\|\Pi_\mathcal {R}(X)\|\!|\leq \gamma$, inequality \eqref{aim-ineq1} holds since
  \(
    \big\|X-\widehat{X}_{\mathcal{F}}\big\|_F=\|X-\Pi_\mathcal {R}(X)\|_F.
  \)
  We next consider the case where $|\!\|\Pi_\mathcal {R}(X)\|\!|>\gamma$. Now we have that
   \begin{align}
    \big\|X-\widehat{X}_{\mathcal{F}}\big\|_F^2
    &= \Big\|X-\Pi_\mathcal {R}(X)+\Pi_\mathcal {R}(X)\Big(1-\frac{\gamma}{|\!\|\Pi_\mathcal {R}(X)\|\!|}\Big)\Big\|_F^2\nonumber\\
    &= \|X-\Pi_\mathcal {R}(X)\|_F^2+\|\Pi_\mathcal {R}(X)\|_F^2\Big(\frac{|\!\|\Pi_\mathcal {R}(X)\|\!|-\gamma}{|\!\|\Pi_\mathcal {R}(X)\|\!|}\Big)^2\nonumber\\
    &\leq \|X-\Pi_\mathcal {R}(X)\|_F^2+\Big(\frac{c_u\|X\!-\!\Pi_\mathcal {R}(X)\|_F\|\Pi_\mathcal {R}(X)\|_F}{|\!\|\Pi_\mathcal {R}(X)\|\!|}\Big)^2\nonumber\\
     &\leq \|X-\Pi_\mathcal {R}(X)\|_F^2+\Big(\frac{c_u\|X\!-\!\Pi_\mathcal {R}(X)\|_F}{c_l}\Big)^2\nonumber
   \end{align}
   where the second equality is by $\langle X-\Pi_\mathcal {R}(X),\Pi_\mathcal {R}(X)\rangle=0$,
   the first inequality is since
   \[
     \gamma<|\!\|\Pi_\mathcal {R}(X)\|\!|\!\le\! |\!\|X\|\!|\!+\!|\!\|X\!-\!\Pi_\mathcal {R}(X)\|\!|
   \!\le\! \gamma\!+\!c_u\|X\!-\!\Pi_\mathcal {R}(X)\|_F,
   \]
  and the second one is due to $c_l\|\Pi_\mathcal {R}(X)\|_F\leq |\!\|\Pi_\mathcal {R}(X)\|\!|$.
  The last inequality implies \eqref{aim-ineq1}.
  We next prove that \eqref{inclusion-ball} holds. Indeed, if $t\in(-\infty,0)$,
  then the inclusion \eqref{inclusion-ball} immediately holds since $\Gamma(t)=\emptyset$.
  Now let $t$ be an arbitrary point from $[0,+\infty)$.
  Let $Z$ be an arbitrary point from $\Gamma(t)$ and define $\widehat{Z}_{\mathcal{F}}$
  as in \eqref{XF}. Then, by noting that $\Gamma(t)\subseteq\Omega$, we have
  \begin{align}\label{temp-ineq211}
    {\rm dist}(Z,\Gamma(0))\le \|Z-\widehat{Z}_{\mathcal{F}}\|_F
    &\le \sqrt{1+c^2_u/c_l^2}\,\|Z-\Pi_\mathcal {R}(Z)\|_F \nonumber\\
    &\le \sqrt{1+c^2_u/c_l^2}\,\|Z-\Pi_\mathcal {R}(Z)\|_*=\sqrt{1+c^2_u/c_l^2}\,t,
   \end{align}
  where the equality is due to $\|Z-\!\Pi_\mathcal {R}(Z)\|_*\!=\sum_{i=\kappa+1}^n\sigma_i(Z)$
  and $Z\in \Gamma(t)$. Equation \eqref{temp-ineq211} means that
  $Z\in \Gamma(0)+\sqrt{1+c^2_u/c_l^2} t\,\mathbb{B}_{\mathbb{R}^{n_1\times n_2}}$.
  Thus, \eqref{inclusion-ball} follows by the arbitrariness of $t$ in $\mathbb{R}$.

  \medskip

  When $|\!\|\cdot\|\!|$ is unitarily invariant, it is easy to check that
  $\Pi_\mathcal {R}(X)\in \Gamma(0)$ for any $X\in \Omega$ (see \cite[Corollary 3.5.9]{HJ91}).
  Then, by letting $\widehat{X}_{\mathcal{F}}=\Pi_\mathcal {R}(X)$ and using the same arguments as above,
  we have the second part of conclusions. The proof is completed.
  \end{proof}

  \medskip

  The result of Proposition \ref{prop-ball} is not trivial when $|\!\|\cdot\|\!|$
  is not unitarily invariant. Taking $|\!\|\cdot\|\!|=\|\cdot\|_\infty$, the infinity norm of matrices,
  since $\frac{1}{\sqrt{n_1n_2}}\|\cdot\|_F\le\!\|\cdot\|_\infty\!\le \!\|\cdot\|_F$, we have
  \begin{equation}
  {\rm dist}(X,\mathcal {F})\leq  \sqrt{1+n_1n_2}\sum_{i=\kappa+1}^n\sigma_i(X)\quad {\rm for}\
   \|X\|_\infty\!\leq \gamma.\nonumber
  \end{equation}

\medskip
 
 {\bf (2) Rank constrained density matrix optimization problems.}
 The feasible set of this class of rank constrained optimization problems takes the following form
  \begin{equation}\label{F-trace}
    \mathcal{F}:=\Big\{X\in\mathbb{H}_{+}^{n}\ |\ {\rm rank}(X)\le\kappa,\ {\rm tr}(X)=1\Big\},
  \end{equation}
  we provide a Lipschitzian bound for the error ${\rm dist}(X,\mathcal {F})$
  with $X\in\mathbb{H}_{+}^{n}$ and ${\rm tr}(X)=1$.
  \begin{proposition}\label{prop-trace}
  Let $\Omega=\big\{X\in\mathbb{H}_{+}^{n}\ |\ {\rm tr}(X)=1\big\}$.
  Then, for any $X\in\Omega$, by assuming that $X$ has the eigenvalue decomposition
  of the form $\sum_{i=1}^n\lambda_i(X)u_iu_i^{\mathbb{T}}$, it holds that
  \begin{equation*}
   {\rm dist}(X,\mathcal {F})\le \sqrt{\|X-\Pi_{\mathcal{R}}(X)\|_F^2+\|X-\Pi_{\mathcal{R}}(X)\|_*^2}\quad{\rm with}\
   \Pi_{\mathcal{R}}(X):=\sum_{i=1}^{\kappa}\lambda_i(X)u_iu_i^{\mathbb{T}}.
  \end{equation*}
  In particular, the inclusion $\Gamma(t)\subseteq \Gamma(0)+\sqrt{2}\,|t|\,\mathbb{B}_{\mathbb{H}^n}$
  holds for any $t\in\mathbb{R}$.
 \end{proposition}
 \begin{proof}
  Let $X$ be an arbitrary point from $\Omega$ with $X=\sum_{i=1}^n\lambda_i(X)u_iu_i^\mathbb{T}$. Define
  \begin{equation}\label{XF-trace}
    \widehat{X}_{\mathcal{F}}
    :=\frac{\Pi_\mathcal {R}(X)}{{\rm tr}(\Pi_\mathcal {R}(X))}.
  \end{equation}
  Notice that ${\rm tr}(\Pi_\mathcal {R}(X))>0$ since $\lambda_1(X)>0$.
  Hence, $\widehat{X}_\mathcal {F}$ is well defined
  and $\widehat{X}_\mathcal {F}\in \mathcal {F}$.
  Thus, by the definition of $\widehat{X}_{\mathcal{F}}$, we have
  \begin{align}
  ({\rm dist}(X,\mathcal {F}))^2\le\|X-\widehat{X}_{\mathcal{F}}\|_F^2
    &= \Big\|X-\Pi_\mathcal {R}(X)+\Pi_\mathcal {R}(X)\Big(1-\frac{1}{{\rm tr}(\Pi_\mathcal {R}(X))}\Big)\Big\|_F^2\nonumber\\
    & = \|X-\Pi_\mathcal {R}(X)\|_F^2 +\|\Pi_\mathcal {R}(X)\|_F^2\Big(1-\frac{1}{{\rm tr}(\Pi_\mathcal {R}(X))}\Big)^2\nonumber\\
    &= \|X-\Pi_\mathcal {R}(X)\|_F^2 +\Big(\frac{\|\Pi_\mathcal {R}(X)\|_F(1-{\rm tr}(\Pi_\mathcal {R}(X)))}{{\rm tr}(\Pi_\mathcal {R}(X))}\Big)^2\nonumber\\
    &\le \|X-\Pi_\mathcal {R}(X)\|_F^2+\|X-\Pi_\mathcal {R}(X)\|_*^2\nonumber
  \end{align}
  where the inequality is due to ${\rm tr}(\Pi_\mathcal {R}(X))\!=\!\|\Pi_\mathcal {R}(X)\|_*\!\geq\! \|\Pi_\mathcal {R}(X)\|_F$
  and $1-{\rm tr}(\Pi_\mathcal {R}(X))=\|X\|_*-\|\Pi_\mathcal {R}(X)\|_*=\|X-\Pi_\mathcal {R}(X)\|_*$.
  Thus, we complete the proof of the inequality.

  \medskip

  Now let $Z$ be an arbitrary point from $\Gamma(t)$. Noting that $Z\in\Gamma(t)\subseteq\Omega$, we have
  \[
    {\rm dist}(Z,\Gamma(0))\le \|Z-\widehat{Z}_{\mathcal{F}}\|_F
    \le \sqrt{2}\|Z-\Pi_\mathcal {R}(Z)\|_*=\sqrt{2}\sum_{i=\kappa+1}^n\sigma_i(Z)=\sqrt{2} t
  \]
  where $\widehat{Z}_{\mathcal{F}}$ is defined as in \eqref{XF-trace},
  and the second equality is due to $Z\in\Gamma(t)$.
  This shows that $Z\in \Gamma(0)+\sqrt{2} t\mathbb{B}_{\mathbb{H}^n}$.
  From the arbitrariness of $t$, the desired inclusion follows.
  \end{proof}

\medskip

 {\bf (3) Rank constrained correlation matrix optimization problems.}
 The feasible set of this class of rank constrained optimization problems takes the following form
 \begin{equation}\label{F-corr}
    \mathcal{F}:=\Big\{X\in\mathbb{H}_{+}^{n}\ |\ {\rm rank}(X)\le\kappa,\ {\rm diag}(X)=e\Big\}
 \end{equation}
 where $e\in\mathbb{R}^n$ is the vector of all ones. The following proposition provides
 a Lipschitzian bound for the error ${\rm dist}(X,\mathcal {F})$ with
 $X\in\mathbb{H}_{+}^{n}$ and ${\rm diag}(X)=e$.
 \begin{proposition}\label{prop-corr}
  Let $\Omega=\{X\in \mathbb{H}^{n}_+\ |\ {\rm diag}(X)=e\}$. Then,
  for any $X\in\Omega$, by assuming that $X$ has the eigenvalue decomposition
  of the form $\sum_{i=1}^n\lambda_i(X)u_iu_i^{\mathbb{T}}$, it holds that
  \begin{equation}\label{ineq-diag}
  {\rm dist}(X,\mathcal {F})\le (1+2n)\|X-\Pi_{\mathcal{R}}(X)\|_*\quad {\rm with}\ \
  \Pi_{\mathcal{R}}(X):=\sum_{i=1}^{\kappa}\lambda_i(X)u_iu_i^{\mathbb{T}}.
  \end{equation}
  In particular, the inclusion $\Gamma(t)\subseteq \Gamma(0)+(1+2n)\,|t|\,\mathbb{B}_{\mathbb{H}^n}$
  holds for any $t\in\mathbb{R}$.
  \end{proposition}
  \begin{proof}
   Let $X$ be an arbitrary point from the set $\Omega$ with $X=\sum_{i=1}^n\lambda_i(X)u_iu_i^\mathbb{T}$. Define
   \begin{equation}\label{XF-corr}
   \widehat{X}_{\mathcal{F}}:=\!\left\{\!
                 \begin{array}{cl}\!
                  D(X)\Pi_\mathcal {R}(X)D(X) & \textrm{if}\ {\rm diag}(\Pi_\mathcal {R}(X))\!>\!0\\
                   ee^\mathbb{T}\ & \textrm{otherwise}
                 \end{array}
                 \right.
  \end{equation}
  with $D(X)={\rm Diag}\Big(\frac{1}{\sqrt{{\rm diag}(\Pi_{\mathcal{R}}(X))}}\Big)$.
  From the expression of $\Pi_{\mathcal{R}}(X)$, it follows that
  \begin{equation}\label{X-mPCA1}
   [{\rm diag}(\Pi_\mathcal {R}(X))]_j \!= \! 1-\!\sum_{i=\kappa+1}^n\!\lambda_i(X)|u_{ij}|^2\geq 0,\quad j\!=\!1,2,\ldots,n
  \end{equation}
  where $u_{ij}$ means the $j$th entry of $u_i$.
  It is easy to check that $\widehat{X}_{\mathcal{F}}\in\mathcal {F}$, which implies that
  ${\rm dist}(X,\mathcal {F})\le \|X-\widehat{X}_{\mathcal{F}}\|_F$.
  Thus, to prove the desired result, we only need to argue that
  \begin{equation}\label{aim-ineq3}
    \|X-\widehat{X}_{\mathcal{F}}\|_F \le  (1+2n)\|X-\Pi_{\mathcal{R}}(X)\|_*.
  \end{equation}
  If there is an index $j$ such that $[{\rm diag}(\Pi_\mathcal {R}(X))]_j=0$,
  by the definition of $\widehat{X}_{\mathcal{F}}$ we have
  \[
    \|X-\widehat{X}_{\mathcal{F}}\|_F=\|X-ee^\mathbb{T}\|_F
    \le\|X\|_F+\|ee^\mathbb{T}\|_F\le n+{n}\le 2n{\textstyle\sum_{i=\kappa+1}^n\lambda_i(X)},
  \]
  where the last inequality is since
  $\sum_{i=\kappa+1}^n\lambda_i(X)\geq\sum_{i=\kappa+1}^n\lambda_i(X)|u_{ij}|^2 =1$.
  Thus, \eqref{aim-ineq3} follows. We next consider the case where ${\rm diag}(\Pi_\mathcal {R}(X))>0$.
  For convenience, we denote $D(X)$ by $D$ and write its $i$th diagonal entry as $D_{ii}$.
  From \eqref{X-mPCA1}, it is clear that $D_{jj}\geq 1$ for all $j$. Notice that
  $(\widehat{X}_{\mathcal{F}})_{kl}=(D\Pi_\mathcal {R}(X)D)_{kl}=D_{kk}D_{ll}(\Pi_\mathcal {R}(X))_{kl}$. Then,
  \begin{align}\label{X-mPCA2}
   \big\|\widehat{X}_{\mathcal{F}}-\Pi_\mathcal {R}(X)\big\|_F^2&=\sum_{k=1}^n\sum_{l=1}^n\big((\widehat{X}_{\mathcal{F}})_{kl}\!-\!(\Pi_\mathcal {R}(X))_{kl}\big)^2 = \sum_{k=1}^n\sum_{l=1}^n\big(D_{kk}D_{ll}\!-\!1\big)^2[(\Pi_\mathcal {R}(X))_{kl}]^2\nonumber\\
   & \leq \max_{1\leq j\leq n}\big(D_{jj}^2-1\big)^2\big\|\Pi_\mathcal {R}(X)\big\|_F^2 \leq n^2\max_{1\leq j\leq n}\big(D_{jj}^2-1\big)^2,
  \end{align}
  where the last inequality is due to $\|\Pi_\mathcal {R}(X)\|_F\leq \|X\|_F\leq n$.
   By using \eqref{X-mPCA2}, we have
  \begin{align}\label{X-mPCA4}
   \|X-\widehat{X}_{\mathcal{F}}\|_F&\leq\|X-\Pi_\mathcal {R}(X)\|_{F}+\|\Pi_\mathcal {R}(X)-\widehat{X}_{\mathcal{F}}\|_{F}\nonumber\\
   &\le \|X-\Pi_\mathcal {R}(X)\|_{F}+n\max_{1\leq j\leq n}\big(D_{jj}^2-1\big)\nonumber\\
   &= \|X-\Pi_\mathcal {R}(X)\|_{F}+n\max_{1\leq j\leq n}\Big(\frac{1}{[{\rm diag}(\Pi_\mathcal {R}(X))]_j}-1\Big)\nonumber\\
   &\leq \|X-\Pi_\mathcal {R}(X)\|_{F}+\frac{n\sum_{i=\kappa+1}^n\!\lambda_i(X)}{\min_{1\leq j\leq n}[{\rm diag}(\Pi_\mathcal {R}(X))]_j}\nonumber\\
   &\le \sum_{i=\kappa+1}^n\!\lambda_i(X)\Big(1+\frac{n}{\min_{1\leq j\leq n}[{\rm diag}(\Pi_\mathcal {R}(X))]_j}\Big),
 \end{align}
 where the third inequality is due to $|u_{ij}|\leq 1$. If $\sum_{i=\kappa+1}^n\lambda_i(X)\leq0.5$,
 equation \eqref{X-mPCA1} implies that $[{\rm diag}(\Pi_\mathcal {R}(X))]_j\ge 0.5$ for all $j$.
 Then by using \eqref{X-mPCA4} we obtain
 \[
   \|X-\widehat{X}_{\mathcal{F}}\|_F\leq (1+2n){\textstyle\sum_{i=\kappa+1}^n\lambda_i(X)}.
 \]
 Now, we assume that $\sum_{i=\kappa+1}^n\lambda_i(X)>0.5$. Since $(\widehat{X}_{\mathcal{F}})_{kl}=D_{kk}D_{ll}(\Pi_\mathcal {R}(X))_{kl}$ and $D_{kk}\geq 1$ for all $k$, we have
 \(
  \|\widehat{X}_{\mathcal{F}}-\Pi_\mathcal {R}(X)\|_F\leq \|\widehat{X}_{\mathcal{F}}\|_F\leq n.
 \)
 Consequently,
 \begin{align}
   \|X-\widehat{X}_{\mathcal{F}}\|_F\leq\|X-\Pi_\mathcal {R}(X)\|_{F}+\|\Pi_\mathcal {R}(X)-\widehat{X}_{\mathcal{F}}\|_{F}
   \leq \|X-\Pi_\mathcal {R}(X)\|_*+ n,\nonumber
 \end{align}
 which along with $\sum_{i=\kappa+1}^n\lambda_i(X)>0.5$ implies that
 $\|X-\widehat{X}_{\mathcal{F}}\|_F\leq (1+2n)\sum_{i=\kappa+1}^n\lambda_i(X)$.
 Thus, we show that inequality \eqref{aim-ineq3} holds.
 The first part of the conclusions follows. By inequality \eqref{aim-ineq3},
 using the same arguments as those for the second part of Proposition \ref{prop-trace},
 we obtain the desired inclusion. The proof is completed.
 \end{proof}

 \medskip

 From Propositions \ref{prop-ball}-\ref{prop-corr}, we see that for the above three class of
 rank constrained problems, there exists a constant $\alpha$ such that the associated
 multifunction $\Gamma$ satisfies
 \[
   \Gamma(t)\subseteq \Gamma(0)+\alpha\,|t|\,\mathbb{B}_{\mathbb{X}}\quad\ \forall t\in\mathbb{R}.
 \]
 By \cite[Definition 1]{Robinson79}, this is actually the upper Lipschitzian
 at $0$ of $\Gamma$ with respect to the set $\mathbb{R}$, which clearly implies
 the calmness of $\Gamma$ at $0$ for each $X\in\Gamma(0)$.

 \medskip

 Next, we establish a local error bound for estimating the distance
 ${\rm dist}(X,\mathcal{F}^*)$ with $X\in \Omega$, under the calmness of $\Gamma$ at $0$
 for each $X\in\Gamma(0)$ and a suitable assumption on $f$.
 \begin{theorem}\label{theorem2-local1}
 Suppose that the multifunction $\Gamma$ in \eqref{Gamma} is calm at $0$ for each $X\in\Gamma(0)$,
 and $f$ is a smooth convex function such that
  \begin{equation}\label{vtheta}
   \vartheta:=\inf_{X\in \mathcal{F}, Y\in \mathcal{F}^*, X\neq Y} \frac{f(X)-f(Y)-\langle \nabla f(Y), X-Y\rangle}{\|X-Y\|_F^2}>0.
 \end{equation}
  Then, there exists a constant $c>0$ such that for any $X\in\Omega$,
  \begin{equation}\label{slocal-bound}
   {\rm dist}(X,\mathcal {F}^*)\le c\sum_{i=\kappa+1}^n\!\sigma_i(X) +\frac{1}{\vartheta}\|\nabla f(\Pi_{\mathcal{F}}(X))-\nabla f(X^*)\|_F,
   \end{equation}
  where $X^*$ is an arbitrary point from $\mathcal{F}^*$.
  When $\mathcal{F}$ takes the form of \eqref{F-ball}, \eqref{F-trace} and \eqref{F-corr}, respectively,
  the constant $c$ can be specified as $c\!=\!\sqrt{1\!+\!c_u^2/c_l^2},\sqrt{2}$ and $1\!+\!2n$, respectively.
 \end{theorem}
 \begin{proof}
  From the definition of $\vartheta$ and the convexity of the function $f$, it follows that
  \begin{align*}
  \vartheta\|X^*-\Pi_{\mathcal{F}}(X)\|_F^2&\leq f(\Pi_{\mathcal{F}}(X))-f(X^*)-\langle \nabla f(X^*), \Pi_{\mathcal{F}}(X)-X^*\rangle\nonumber\\
   &\leq \big\langle \nabla f(\Pi_{\mathcal{F}}(X))-\nabla f(X^*), \Pi_{\mathcal{F}}(X)-X^*\big\rangle\nonumber\\
   &\leq \|\nabla f(\Pi_{\mathcal{F}}(X))-\nabla f(X^*)\|_F\|\Pi_{\mathcal{F}}(X)-X^*\|_F,
  \end{align*}
  which implies that
  \(
   \|X^*-\Pi_{\mathcal{F}}(X)\|_F\leq  \frac{1}{\vartheta}\|\nabla f(\Pi_{\mathcal{F}}(X))-\nabla f(X^*)\|_F.
  \)
  Thus, we have that
  \begin{align}\label{theorem211-ball11}
  {\rm dist}(X, \mathcal {F}^*) &\le \|X-X^*\|_F \le \|X-\Pi_{\mathcal{F}}(X)\|_F + \|X^*-\Pi_{\mathcal{F}}(X)\|_F\nonumber\\
  &\leq c\sum_{i=\kappa+1}^n\!\sigma_i(X) +\frac{1}{\vartheta}\|\nabla f(\Pi_{\mathcal{F}}(X))-\nabla f(X^*)\|_F,
 \end{align}
  where the last inequality is using Theorem \ref{theorem2.1}.
  By Propositions \ref{prop-ball}-\ref{prop-corr}, when $\mathcal{F}$ is given by \eqref{F-ball}, \eqref{F-trace} and \eqref{F-corr},
  $c$ can be specified as $\sqrt{1\!+\!c_u^2/c_l^2},\sqrt{2}$ and $1\!+\!2n$, respectively.
 \end{proof}
 \begin{remark}
  {\bf (a)} The assumption $\vartheta>0$ in Theorem \ref{theorem2-local1} is a kind of restricted strong
  convexity condition, which extends the sparse reconstruction condition used for
  the analysis of sparse constrained optimization problems (see, e.g., \cite{Bahmani-GSCO, Yuan-GHT}).
  One can check that
  \begin{align}\label{RSC1}
    0<\inf_{1\le{\rm rank}([X\ Y])\le 2\kappa, X\neq Y} \frac{f(X)-f(Y)-\langle \nabla f(Y), X-Y\rangle}{\|X-Y\|_F^2}
  \end{align}
  implies $\vartheta>0$, where \eqref{RSC1} is an extension of sparse reconstruction condition in \cite{Bahmani-GSCO,Yuan-GHT}.

  \medskip
  \noindent
  {\bf (b)} When $f$ is twice continuously differentiable, it is not difficult to check that
  \[
    \frac{1}{\vartheta}\|\nabla f(\Pi_{\mathcal{F}}(X))-\nabla f(X^*)\|_F
    \le \frac{\vartheta_{\rm max}}{\vartheta_{\rm min}}\|\Pi_{\mathcal{F}}(X)-X^*\|_F
  \]
  where
  $\vartheta_{\rm min}=\min_{X\in\mathcal{F}}\lambda_{\rm min}(\nabla^2f(X))$
  and $\vartheta_{\rm max}=\max_{X\in\mathcal{F}}\lambda_{\rm max}(\nabla^2f(X))$.
  This shows that the bound in \eqref{slocal-bound} is related to the condition number
  of the Hessian matrix $\nabla^2f(X)$ restricted over the set $\mathcal{F}$,
  which is clearly smaller than the condition number of $\nabla^2f(X)$ restricted
  over the set $\mathcal{R}$. Together with Theorem \ref{theorem2.1},
  we actually have that
  \[
    \frac{1}{\vartheta}\|\nabla f(\Pi_{\mathcal{F}}(X))-\nabla f(X^*)\|_F
    \le \frac{\vartheta_{\rm max}}{\vartheta_{\rm min}}\big(c\sum_{i=\kappa+1}^n\!\sigma_i(X)+\|X\!-\!X^*\|_F\big)\ \ \forall\,X\in \Omega.
  \]
  \end{remark}

  \subsection{Global error bounds}\label{subsec2.2}

  Generally, for a nonconvex feasible set, it is almost impossible to get a global error
  bound for estimating the distance of any point to the feasible set (see \cite{Pang97-Exp}).
  However, under the calmness of $\Gamma$ at $0$ for each $X\in\Gamma(0)$, we can establish
  a global error bound for estimating the distance from any $X\in\mathbb{X}$ to the feasible set
  $\mathcal{F}$ and the solution set $\mathcal{F}^*$.
 \begin{theorem}\label{theorem1-global}
  Suppose that the multifunction $\Gamma$ in \eqref{Gamma} is calm at $0$ for each $X\in\Gamma(0)$.
  Then, there exists a constant $c>0$ such that for any $X\in\mathbb{X}$,
  \begin{equation}\label{ineq1-global}
   {\rm dist}(X,\mathcal {F})\leq (1+c\sqrt{n}){\rm dist}(X,\Omega) +c\sum_{i=\kappa+1}^n\!\sigma_i(X).
  \end{equation}
  When $\mathcal{F}$ takes the form of \eqref{F-ball}, \eqref{F-trace} and \eqref{F-corr}, respectively,
  the constant $c$ can be specified as $\!\sqrt{1+c_u^2/c_l^2},\sqrt{2}$ and $1+2n$, respectively.
  In particular, for the set $\mathcal{F}$ in \eqref{F-ball}, when $|\!\|\cdot\|\!|$ is unitarily invariant,
  one may take $c=1$.
 \end{theorem}
 \begin{proof}
  Let $X$ be an arbitrary point from $\mathbb{X}$. Then, the following inequalities hold
  \begin{align*}
   {\rm dist}(X,\mathcal{F})&\leq\|X-\Pi_{\mathcal {F}}(\Pi_\Omega(X))\|_F\nonumber\\
   &\leq \|X-\Pi_\Omega(X)\|_F+\|\Pi_\Omega(X)-\Pi_{\mathcal {F}}(\Pi_\Omega(X))\|_F\nonumber\\
   &\leq \|X-\Pi_\Omega(X)\|_F + c\big\|\Pi_\Omega(X)-\Pi_{\mathcal {R}}(\Pi_\Omega(X))\big\|_*\nonumber\\
   & = \|X-\Pi_\Omega(X)\|_F + c\min_{{\rm rank}(Y)\leq \kappa}\big\|\Pi_\Omega(X)-Y\big\|_*\nonumber\\
   & \leq \|X-\Pi_\Omega(X)\|_F + c\big(\|\Pi_\Omega(X)-X\|_*+\|X-\Pi_\mathcal {R}(X)\|_*\big)\nonumber\\
    & \leq \|X-\Pi_\Omega(X)\|_F + c\big(\sqrt{n}\|\Pi_\Omega(X)-X\|_F+\|X-\Pi_\mathcal {R}(X)\|_*\big)\nonumber\\
   & = (1+c\sqrt{n}){\rm dist}(X,\Omega)+ c\|X-\Pi_\mathcal {R}(X)\|_*,
  \end{align*}
  where the third inequality is using Theorem \ref{theorem2.1} and the fourth one
   is due to the inequality $\min_{{\rm rank}(Y)\le \kappa}\|\Pi_\Omega(X)-Y\|_*\le\|\Pi_\Omega(X)-\Pi_\mathcal {R}(X)||_*$.
   The proof is completed.
  \end{proof}
 \begin{remark}
  For the term ${\rm dist}(X,\Omega)$ in \eqref{ineq1-global},
  when the set $\mathcal{F}$ is given by \eqref{F-ball}, we have
  \[
   {\rm dist}(X,\Omega)\leq \frac{1}{c_l}\max(|\!\|X\|\!|-\gamma,0)\  \ {\rm for\ any}\ X\in \mathbb{R}^{n_1\times n_2};
  \]
  when $\mathcal{F}$ takes the form of \eqref{F-trace},
  using \cite[Theorem 2.1]{Zhang} with $\mathcal{L}=\{Y\in \mathbb{H}^n\ |\ {\rm tr}(Y)=0\}$,
  $B=\frac{1}{n}I$ and $\mathcal {K}=\mathbb{H}^n_+$ shows that there exists a constant $\nu\ge 1$ such that
  \begin{align}
    {\rm dist}(X,\Omega)&\le \nu\big({\rm dist}(X,B+\mathcal{L})+\|X-\Pi_{\mathbb{H}^n_+}(X)\|_F\big)\nonumber\\
     &= \frac{\nu}{\sqrt{n}}\big|1-{\rm tr}(X)\big|+\nu\|X-\Pi_{\mathbb{H}^n_+}(X)\|_F\  \ {\rm for\ any}\ X\in \mathbb{H}^n;\nonumber
  \end{align}
  and when $\mathcal{F}$ takes the form of \eqref{F-corr},
  using \cite[Theorem 2.1]{Zhang} with $\mathcal{L}=\{Y\in \mathbb{H}^n\ |\ {\rm diag}(Y)=0\}$,
  $B=I$ and $\mathcal {K}=\mathbb{H}^n_+$ yields that there exists a constant $\nu\ge 1$ such that
  \begin{align}
    {\rm dist}(X,\Omega)&\le \nu\big({\rm dist}(X,B+\mathcal{L})+\|X-\Pi_{\mathbb{H}^n_+}(X)\|_F\big)\nonumber\\
    &= \nu\big(\|e-{\rm diag}(X)\|+\|X-\Pi_{\mathbb{H}^n_+}(X)\|_F\big)\  \ {\rm for\ any}\ X\in \mathbb{H}^n.\nonumber
  \end{align}
 \end{remark}

 Now following the same arguments as for inequality \eqref{theorem211-ball11} and
 using Theorem \ref{theorem1-global} to bound $\|X-\Pi_{\mathcal{F}}(X)\|_F$,
 we can establish a global error bound for estimating the distance from any point
 $X\in \mathbb{X}$ to the solution set $\mathcal{F}^*$ under the assumptions of
 Theorem \ref{theorem2-local1}.
 \begin{theorem}\label{theorem2-global}
  Suppose that the multifunction $\Gamma$ in \eqref{Gamma} is calm at $0$ for each $X\in\Gamma(0)$,
  and $f$ is a smooth convex function such that $\vartheta>0$, where $\vartheta$ is defined by
  \eqref{vtheta}. Then, there exists a constant $c>0$ such that for any $X\in \mathbb{X}$,
  \begin{align}
   {\rm dist}(X,\mathcal {F}^*)\!\le\! (1\!+\!c\sqrt{n}){\rm dist}(X,\Omega) +c\sum_{i=\kappa+1}^n\!\sigma_i(X)
   \!+\!\frac{1}{\vartheta}\|\nabla f(\Pi_{\mathcal{F}}(X))-\nabla f(X^*)\|_F,\nonumber
   \end{align}
  where $X^*$ is an arbitrary point from $ \mathcal{F}^*$.
  When $\mathcal{F}$ takes the form of \eqref{F-ball}, \eqref{F-trace} and \eqref{F-corr},
  respectively, the constant $c$ can be specified as $\sqrt{1+c_u^2/c_l^2},\sqrt{2}$ and $1+2n$,
  respectively.
  \end{theorem}

 \section{Applications of local error bounds}\label{sec3}

  This section illustrates the applications of the local Lipschitzian type error bounds
  in establishing the exact penalty for the rank constrained optimization problem \eqref{prob},
  and deriving an error bound for a multi-stage convex relaxation approach to problem \eqref{prob} in
 which $\Omega$ is a ball set, a density matrix set and a correlation matrix set, respectively.
  \subsection{Exact penalty for problem \eqref{prob}}\label{sec3.1}

  With the help of the local error bounds in Theorem \ref{theorem2.1}, we show that
  \eqref{penalty-rc} is an exact penalty problem of \eqref{prob} in the sense that
  the global optimal solution set of \eqref{penalty-rc} coincides with that of problem \eqref{prob}
  when the penalty parameter is over a threshold. This result is stated in the following theorem.
  Its proof technique is similar to that of \cite[Theorem 2.1.2]{Luo96}. Notice that
  the latter focuses on the subanalytic compact set and employs an error bound derived by
  the Lojasiewicz' inequality for subanalytic function, but our result is stated for
  a general compact set $\Omega\cap\mathcal{R}$.
  For completeness, we here include the proof.
  \begin{theorem}\label{epenalty}
   Suppose that the multifunction $\Gamma$ in \eqref{Gamma} is calm at $0$ for each $X\in\Gamma(0)$. Let $L>0$ be
   the Lipschitz constant of the function $f$ over $\Omega$. Then, when $\rho>c L$ with
   $c$ same as the one in Theorem \ref{theorem2.1}, the following statements hold.
   \begin{description}
   \item[(a)] The global optimal solution set of \eqref{prob} coincides with that of \eqref{penalty-rc};

   \item[(b)] $\overline{X}\in\mathcal{F}$ is a local optimal solution of \eqref{prob}
              if and only if $\overline{X}$ is locally optimal to \eqref{penalty-rc}.
   \end{description}
  \end{theorem}
 \begin{proof}
  (a) Since $\Omega$ is compact and $f$ is continuous, the optimal solution set of
  \eqref{penalty-rc} is nonempty. Let $X^*$ be an arbitrary optimal
  solution of \eqref{prob}. Then, for any $X\in\Omega$,
  \begin{align}\label{exp-eb11}
   f(X)+\rho\sum_{i=\kappa+1}^n\!\sigma_i(X)
  &\geq f(X)+c L\sum_{i=\kappa+1}^n\!\sigma_i(X)\geq f(X)+L\|X-\Pi_{\mathcal{F}}(X)\|_F\nonumber\\
  &\geq f(\Pi_{\mathcal{F}}(X)) \geq f(X^*)=f(X^*)+\rho\sum_{i=\kappa+1}^n\!\sigma_i(X^*)
  \end{align}
  where the second inequality is due to Theorem \ref{theorem2.1}.
  The last inequality shows that $X^*$ is an optimal solution of \eqref{penalty-rc}.
  Now, let $\overline{X}$ be an arbitrary optimal solution of \eqref{penalty-rc}.
  From the feasibility of $X^*$ to \eqref{penalty-rc} and Theorem \ref{theorem2.1},
  it follows that
  \begin{align}\label{exp-eb21}
  f(X^*)&\ge f(\overline{X})+\rho\sum_{i=\kappa+1}^n\!\sigma_i(\overline{X})
   \geq f(\overline{X})+c L\sum_{i=\kappa+1}^n\!\sigma_i(\overline{X})\nonumber\\
  &\geq f(\overline{X})+L\|\overline{X}-\Pi_{\mathcal{F}}(\overline{X})\|_F
   \geq f(\Pi_{\mathcal{F}}(\overline{X})) \geq f(X^*).
  \end{align}
  This means that $f(X^*)\!=\!f(\overline{X})+\rho\sum_{i=\kappa+1}^n\!\sigma_i(\overline{X})$
  and $\rho\sum_{i=\kappa+1}^n\!\sigma_i(\overline{X})\!=\!c L\sum_{i=\kappa+1}^n\!\sigma_i(\overline{X})$.
  Since $\rho\!>\!c L$, we must have $\sum_{i=\kappa+1}^n\!\sigma_i(\overline{X})\!=\!0$,
  and then $\overline{X}\!\in\! \mathcal {F}$. Along with $f(X^*)\!=\!f(\overline{X})$,
  it follows that $\overline{X}$ is an optimal solution of problem \eqref{prob}.
  Thus, part (a) follows.

   \medskip
  \noindent
  (b) Since $\overline{X}\in \mathcal {F}$, it is easy to show that if $\overline{X}$ is
  a local optimal solution of \eqref{penalty-rc}, then $\overline{X}$ is locally optimal of \eqref{prob}.
  Now, we assume that $\overline{X}$ is a local optimal solution of \eqref{prob}.
  Then there exists a neighborhood of $\overline{X}$, denoted by $\mathcal {N}(\overline{X}, \overline{\varepsilon})$
  with $\overline{\varepsilon}>0$, such that
  \begin{align}\label{the1-sun-local1}
  f(\overline{X}) \leq f(X)\ {\rm for\ all}\ X\in \mathcal {N}(\overline{X}, \overline{\varepsilon})\cap \mathcal {F}.
  \end{align}
  Notice that $\sum_{i=\kappa+1}^n\!\sigma_i(X)$ is continuous and $\sum_{i=\kappa+1}^n\!\sigma_i(\overline{X})=0$.
  Hence, there must exist $\widetilde{\varepsilon}>0$ such that
  $\sum_{i=\kappa+1}^n\!\sigma_i(X)<\frac{\overline{\varepsilon}}{2c}$ for all $X\in\mathcal {N}(\overline{X}, \widetilde{\varepsilon})$.
  In the following, we show that for any $X\in \mathcal {N}(\overline{X}, \widehat{\varepsilon})\cap \Omega$
  where $0<\widehat{\varepsilon}\leq \min(\frac{\overline{\varepsilon}}{2},\widetilde{\varepsilon})$, it holds that
  \begin{align}
  f(\overline{X}) + \rho\sum_{i=\kappa+1}^n\!\sigma_i(\overline{X})=f(\overline{X})\leq f(X)+\rho\sum_{i=\kappa+1}^n\!\sigma_i(X).\nonumber
  \end{align}
  Notice that for any $X\!\in\! \mathcal {N}(\overline{X}, \widehat{\varepsilon})\!\cap\! \Omega$, by Theorem \ref{theorem2.1} and $\|X\!-\!\overline{X}\|_F\!\leq\!\widehat{\varepsilon}$ we have
  \begin{align}
  \|\Pi_{\mathcal{F}}(X)-\overline{X}\|_F
  &\leq \|\Pi_{\mathcal{F}}(X)-X\|_F+\|X-\overline{X}\|_F
   \leq c\sum_{i=\kappa+1}^n\!\sigma_i(X)+\widehat{\varepsilon}\leq \overline{\varepsilon}.\nonumber
  \end{align}
  Thus, it holds $\Pi_{\mathcal{F}}(X)\subseteq\mathcal {N}(\overline{X}, \overline{\varepsilon})\cap\mathcal {F}$.
  This, together with inequality \eqref{the1-sun-local1}, yields that
  \[
    f(\overline{X})=f(\overline{X}) + \rho\sum_{i=\kappa+1}^n\!\sigma_i(\overline{X})
    \leq f(\Pi_{\mathcal{F}}(X))\leq f(X)+\rho\sum_{i=\kappa+1}^n\!\sigma_i(X)\quad \forall\ X\in \mathcal {N}(\overline{X}, \widehat{\varepsilon})\cap\Omega,\nonumber
  \]
  where the last inequality is due to \eqref{exp-eb11}.
  Then $\overline{X}$ is a local optimal solution of \eqref{penalty-rc}.
  \end{proof}

  \subsection{Error bound for a multi-stage convex relaxation method}\label{sec3.2}

  Let $\Omega$ be one of the sets in Propositions \ref{prop-ball}-\ref{prop-corr}
  and $f$ be a convex function. By Theorem \ref{epenalty}, when $\rho>cL$ with $L>0$
  being the Lipschitz constant of $f$ over $\Omega$, problem \eqref{penalty-rc}, i.e.
  \begin{align}\label{prob2}
   \min_{X\in \Omega}\Big\{f(X)+\rho(\|X\|_*-\|X\|_\kappa)\Big\}
  \end{align}
  has the same global optimal solution set as that of \eqref{prob}, where $\|X\|_{\kappa}$ means
  the Ky Fan $\kappa$-norm, i.e., the sum of the $\kappa$-largest singular values of $X$. Motivated by this,
  at the current iterate $X^{k-1}$ we replace the concave function $-\|X\|_\kappa$
  in \eqref{prob2} by a linear majorization function $-\langle W^{k-1},X\rangle$,
  and transform the solution of \eqref{prob} or its exact penalty problem \eqref{prob2}
  into the solution of a sequence of convex  minimization problems.
  This leads to a multi-stage convex relaxation approach to problem \eqref{prob}.
  The idea of replacing $-\|X\|_\kappa$ by the linear majorization function
  $-\langle W^{k-1},X\rangle$ first appears in the majorized penalty approach proposed
  in \cite{GS-Major}, where $W^{k-1}\in\partial \|\cdot\|_\kappa(X^{k-1})$ is used.
  Here we consider the multi-stage convex relaxation approach where the linear
  majorization function $-\langle W^{k-1},X\rangle$ is given with
  $W^{k-1}\in\partial \|\cdot\|_\kappa(\widehat{X}_{\mathcal{F}}^{k-1})$,
  where $\widehat{X}_{\mathcal{F}}^{k-1}$ is defined by \eqref{XF}, \eqref{XF-trace}
  and \eqref{XF-corr} respectively with $X=X^{k-1}$. Also, different from
  the majorized penalty approach \cite{GS-Major}, our multi-stage convex relaxation
  approach is using $f$ itself instead of its majorization function.
  \bigskip

  \setlength{\fboxrule}{0.8pt}
  \noindent
  \fbox{
  \parbox{0.96\textwidth}
  {{\bf Multi-stage convex relaxation approach for problem (\ref{prob})}\

   \noindent
   \begin{description}
   \item[(S.0)] Let $\rho_0>0$ be given. Choose a starting point $X^0\in\mathbb{X}$.
                Set $k:=0$.

   \item[(S.1)] Compute $\widehat{X}_\mathcal {F}^{k}$ by \eqref{XF}, \eqref{XF-trace} and \eqref{XF-corr} with $X=X^{k}$
                and the associated $\Omega$.

   \item[(S.2)] Seek $W^{k}\in \partial \|\cdot\|_\kappa(\widehat{X}^{k}_\mathcal {F})$ where
                $\partial \|\cdot\|_\kappa$ is the subdifferential map of $\|\cdot\|_\kappa$.

   \item[(S.3)] Find $X^{k+1}\in\mathop{\arg\min}_{X\in\Omega} \Big\{f(X)+\rho_k(\|X\|_*-\langle W^{k}, X\rangle)\Big\}$.

  \item[(S.4)]  Set $\rho_{k+1}:=\tau_k\rho_k$ with $\tau_k\geq 1$. Let $k\leftarrow k+1$, and go to Step (S.1).

  \end{description}
   }
   }

 \bigskip

  Assume that $\widehat{X}^k_\mathcal {F}$ for $k\ge 0$ has the SVD of the form
  $\widehat{X}^k_\mathcal {F}=\sum_{i=1}^n\sigma_i(\widehat{X}^k_\mathcal {F})u_i^k(v_i^k)^\mathbb{T}$ and
  let $U_1^k=[u_1^k\ u_2^k\ \cdots\ u_{\kappa}^k]$ and $V_1^k=[v_1^k\ v_2^k\ \cdots\ v_{\kappa}^k]$.
  Then,
  \(
   W^k=U_1^k(V_1^k)^\mathbb{T}\in\partial \|\cdot\|_\kappa(\widehat{X}^k_\mathcal {F}).
  \)
  The above multi-stage convex relaxation approach generates the sequences
  $\{\widehat{X}_\mathcal {F}^{k}\}_{k\ge 1}\subseteq\mathcal{F}$
  and $\{X^k\}_{k\ge 1}\subseteq\Omega$. In the following, we bound the distance
  from $\widehat{X}_\mathcal {F}^{k}$ (respectively, $X^k$) to the solution set $\mathcal {F}^*$
  with a bound sequence nonincreasing as the number of stages.
  \begin{theorem}\label{eb-mlsp22}
   Suppose that $f$ is
   a nonnegative smooth convex function with Lipschitz continuous gradient
   over $\Omega$ such that $\vartheta$ defined by \eqref{vtheta} is positive.
   Let $\{\widehat{X}_{\mathcal{F}}^k\}$ and $\{X^k\}$ be the sequences generated by
   the multi-stage convex relaxation approach with $\rho_0\!>\! \max\big(cL,f(\widehat{X}_{\mathcal{F}}^0),\frac{c^2(\vartheta+\overline{L})^2}{4\overline{L}}\big)$,
   where $L>0$ and $\overline{L}>0$ are the Lipschitz constants of $f$ and $\nabla f$
   over $\Omega$, respectively, and $c$ is same as the one in Theorem \ref{theorem2-local1}. Then,
   \begin{align}\label{results-eb-mlsp22211}
   {\rm dist}(\widehat{X}_{\mathcal{F}}^k,\mathcal {F}^*)
   \le \frac{M\!+\!\sqrt{M^2\!+\!4\vartheta(f(\widehat{X}_{\mathcal{F}}^k)\!-\!f(X^*))}}{2\vartheta}\leq\cdots
   \le \frac{M\!+\!\sqrt{M^2\!+\!4\vartheta(f(\widehat{X}_{\mathcal{F}}^0)\!-\!f(X^*))}}{2\vartheta}
   \end{align}
   where $M\!:=\!{\displaystyle\max_{Z\in\mathcal{F}^*}}\|\nabla f(Z)\|$ and $X^*\!\in\!\mathcal{F}^*$, and with
    $\rho_{-1}:=cL$, it holds that
   \begin{align}\label{results-eb-mlsp22}
   {\rm dist}(X^k,\mathcal {F}^*)
   \le \Xi^k\leq \Xi^{k-1}\leq\cdots\leq\Xi^1\leq\Xi^0\quad\ {\rm for}\ k\ge 1,
   \end{align}
   where \(
   \Xi^l\!:=\!\frac{2\sqrt{2\overline{L}}}{\vartheta}\sqrt{f(X^l)\!+\!\rho_{l-1}\sum_{i=\kappa+1}^n\sigma_i(X^l)\!+\!f(X^*)}.
   \)
  \end{theorem}
 \begin{proof}
  We first argue that the sequence $\{f(\widehat{X}_{\mathcal{F}}^k)\}_{k\geq 0}$ is nonincreasing.
  Indeed, from the proof of Propositions \ref{prop-ball}-\ref{prop-corr} we have
  $\|X^k-\widehat{X}^k_{\mathcal {F}}\|_F\leq c\sum_{i=\kappa+1}^n\sigma_i(X^k)$.
  This implies that
 \begin{align}\label{ebxf-mlsp222}
  f(X^k)+\rho_{k-1}\!\sum_{i=\kappa+1}^n\sigma_i(X^k)
  \ge f(X^k)+cL\!\sum_{i=\kappa+1}^n\sigma_i(X^k)
  \ge f(\widehat{X}^k_{\mathcal {F}}),\ k\geq 0.
 \end{align}
 Since $\|X\|_{\kappa}\!=\!\sum_{i=1}^\kappa\sigma_i(X)\!=\!\sup_{\|W\|\leq 1,\ {\rm rank}(W)\leq \kappa}\langle W, X\rangle$
 for any $X\!\in\! \mathbb{X}$, it follows that
 \begin{equation}\label{eb-mlsp221-0}
  \sum_{i=\kappa+1}^n\sigma_i(X^k)\!=\!\|X^k\|_*\!-\!\sup_{\|W\|\leq 1,{\rm rank}(W)\leq \kappa}\langle W, X^k\rangle\leq \|X^k\|_*\!-\!\langle W^{k-1}, X^k\rangle\
  {\rm for}\ k\ge 1.
 \end{equation}
 From equation \eqref{eb-mlsp221-0} and the definitions of $X^k$ and
 $X_{\mathcal{F}}^{k-1}\in\Omega$ for $k\ge 1$, we have that
 \begin{align}\label{eb-mlsp222}
 f(X^k)+\rho_{k-1}\sum_{i=\kappa+1}^n\sigma_i(X^k)
 \leq f(X^k)+\rho_{k-1}\big(\|X^k\|_*-\langle W^{k-1}, X^k\rangle\big)\le f(\widehat{X}^{k-1}_\mathcal {F}),
 \end{align}
 where the last inequality is due to
 $\langle W^{k-1}, \widehat{X}^{k-1}_\mathcal {F}\rangle\!-\!\|\widehat{X}^{k-1}_\mathcal {F}\|_\kappa\!=\!0$,
 implied by $W^{k-1}\!\in\! \partial \|\cdot\|_\kappa(\widehat{X}^{k-1}_\mathcal {F})$.
 From \eqref{ebxf-mlsp222} and \eqref{eb-mlsp222},
 $\{f(\widehat{X}^k_{\mathcal {F}})\}_{k\geq 0}$ is nonincreasing.
 By the definition of $\vartheta$,
 \[
   \vartheta\|\widehat{X}_{\mathcal{F}}^k\!-\!X^*\|_F^2
   \le f(\widehat{X}_{\mathcal{F}}^k)-f(X^*)-\langle \nabla f(X^*),\widehat{X}_{\mathcal{F}}^k-X^*\rangle
   \le f(\widehat{X}_{\mathcal{F}}^k)-f(X^*)+M\|\widehat{X}_{\mathcal{F}}^k-X^*\|,
 \]
 which implies that
 \(
   \|\widehat{X}_{\mathcal{F}}^k\!-\!X^*\|_F
   \le \frac{1}{2\vartheta}\Big(M+\sqrt{M^2+4\vartheta(f(\widehat{X}_{\mathcal{F}}^k)-f(X^*))}\Big).
 \)
 This, along with the nonincreasing of the sequence $\{f(\widehat{X}^k_{\mathcal {F}})\}_{k\geq 0}$,
 yields the desired result in \eqref{results-eb-mlsp22211}.

 \medskip

 Notice that $X^k\in \Omega$ for all $k\ge 1$. Therefore, from Theorem \ref{theorem2-local1} it follows that
 \begin{align}\label{eb-mlsp221}
 {\rm dist}(X^k,\mathcal {F}^*)
  &\le c\sum_{i=\kappa+1}^n\sigma_i(X^k)\!+\!\frac{1}{\vartheta}\big(\|\nabla f(\Pi_{\mathcal{F}}(X^k))\!-\!\nabla f(X^k)\|_F
     \!+\!\|\nabla f(X^k)\!-\!\nabla f(X^*)\|_F\big)\nonumber\\
  &\le c\sum_{i=\kappa+1}^n\sigma_i(X^k)+\frac{\overline{L}}{\vartheta}\big\|\Pi_{\mathcal{F}}(X^k)-X^k\big\|_F
     +\frac{1}{\vartheta}\|\nabla f(X^k)-\nabla f(X^*)\|_F\nonumber\\
  &\le c\sum_{i=\kappa+1}^n\sigma_i(X^k)+\frac{\overline{L}}{\vartheta}\big\|\Pi_{\mathcal{F}}(X^k)-X^k\big\|_F
     +\frac{2\sqrt{\overline{L}}}{\vartheta}\sqrt{f(X^k)+f(X^*)}\nonumber\\
  & \leq c\Big(1+\frac{\overline{L}}{\vartheta}\Big)\sum_{i=\kappa+1}^n\sigma_i(X^k)
   +\frac{2\sqrt{\overline{L}}}{\vartheta}\sqrt{f(X^k)+f(X^*)}\nonumber\\
  & \le \frac{2\sqrt{\overline{L}}}{\vartheta}\Big[\frac{c(\vartheta+\overline{L})}{2\sqrt{\overline{L}}}\sqrt{{\textstyle\sum_{i=\kappa+1}^n}\sigma_i(X^k)}
   +\sqrt{f(X^k)+f(X^*)}\Big]\nonumber\\
  & \le \frac{2\sqrt{2\overline{L}}}{\vartheta}\sqrt{f(X^k)+\rho_{k-1}{\textstyle\sum_{i=\kappa+1}^n}\sigma_i(X^k)+f(X^*)}
   \nonumber
 \end{align}
 where the second inequality is using the Lipschitz continuity of $\nabla f$ over $\Omega$, the third one is using \cite[Equation (2.1.10)]{Nes04} implied by the assumption
 that $f$ is a nonnegative smooth convex function with Lipschitz continuous gradient, the fifth one is using $\sum_{i=\kappa+1}^n\sigma_i(X^k)<1$ implied by $\rho_{k-1}> f(\widehat{X}_{\mathcal{F}}^0)\geq f(\widehat{X}_{\mathcal{F}}^k)$ and \eqref{eb-mlsp222},
 and the last one is due to $\rho_{k-1}>\frac{c^2(\vartheta+\overline{L})^2}{4\overline{L}}$.
 Thus, to establish inequality \eqref{results-eb-mlsp22}, it suffices to show that
 $\Xi^k\!\leq\! \Xi^{k-1}\!\leq\!\cdots\!\leq\!\Xi^0$ for $k\geq 1$. Indeed,
 by using equations \eqref{ebxf-mlsp222} and \eqref{eb-mlsp222}, we have that
 \begin{align}
 f(X^k)+\rho_{k-1}\sum_{i=\kappa+1}^n\sigma_i(X^k)\le f(\widehat{X}^{k-1}_\mathcal {F})\leq f(X^{k-1})+\rho_{k-2}\sum_{i=\kappa+1}^n\sigma_i(X^{k-1})\ \ {\rm for}\ k\!\ge\! 1.\nonumber
 \end{align}
 This implies that the sequence $\{f(X^k)\!+\!\rho_{k-1}\sum_{i=\kappa+1}^n\sigma_i(X^k)\}_{k\geq 0}$
 is nonincreasing. By the definition of $\Xi^k$, it follows that $\Xi^k\!\leq\! \Xi^{k-1}\!\leq\!\cdots\!\leq\Xi^1\!\leq\!\Xi^0$.
 The proof is completed.
 \end{proof}

 \medskip

 From the proof of Theorem \ref{eb-mlsp22}, we have $\sum_{i=\kappa+1}^n\sigma_i(X^k)\le f(\widehat{X}^{k-1}_{\mathcal {F}})/\rho_{k-1}$.
 By noting that $\sum_{i=\kappa+1}^n\sigma_i(X^k)=0$ means $X^k\in\mathcal{F}$, this shows that
 $X^k$ is an approximate feasible solution to \eqref{prob}, and the infeasibility violation
 becomes smaller as the number of stages increases, since the sequence $\{f(\widehat{X}^{k-1}_{\mathcal {F}}\}$
 is nonincreasing and $\rho_k$ is nondecreasing.

 \section{Conclusions}\label{sec5}

  In this paper we have provided a sufficient and necessary condition to establish
  the Lipschitzian type error bounds for estimating the distance from any $X\in \Omega$
  to the feasible set $\mathcal{F}$ of the rank constrained optimization problem \eqref{prob},
  and showed that this condition is specially satisfied by three classes of common $\Omega$.
  With the help of this result and the error bound for the convex feasibility system,
  we also derived the global error bound for estimating the distance from any $X\in\mathbb{X}$
  to $\mathcal{F}$. Under an additional suitable restricted strong convexity for the objective function $f$,
  the error bounds for estimating the distance from $X\in \Omega$ (or $X\in\mathbb{X}$) to
  the solution set $\mathcal{F}^*$ are also derived. To illustrate the applications of these error bounds,
  we have showed that the penalty problem yielded by moving the rank constraint ${\rm rank}(X)\le\kappa$
  into the objective is exact, which affirmatively answers the open question proposed in \cite{GS-Major}
  about whether the penalty problem (32) there is exact or not, and provided the error bound
  of the iterates generated by a multi-stage convex relaxation approach to \eqref{prob}
  with three classes of special $\Omega$.

  \medskip

  To the best of our knowledge, this paper is the first one to touch the error bound and
  exact penalty for the NP-hard low-rank optimization problems. Clearly, the error bound
  and exact penalty results also hold for the corresponding classes of zero norm constrained
  optimization problems, such as the sparse principal components analysis and the sparse
  portfolio selection problems. Our future research work will focus on the applications of
  these Lipschitzian error bounds to the convergence and iteration complexity of algorithms
  for low-rank constrained optimization problems, especially the error bound of the iterates
  yielded by the majorized penalty approach \cite{GS-Major}.

 \bigskip
 \noindent
 {\large\bf Acknowledgements.} The authors would like to thank anonymous referees
 and the associated editor for their helpful suggestions on the revision of the original manuscript.

 \end{document}